\documentclass[psamsfonts]{amsart}
\usepackage{amssymb}

\usepackage{amsmath,amstext,amssymb,amsopn,amsthm}
\usepackage{amsmath,amssymb,amsthm}
\usepackage[mathscr]{eucal}
\usepackage{bbm}
\usepackage{amssymb}
\usepackage{amsmath}
\usepackage{color}
\usepackage{tikz}
\usetikzlibrary{arrows}
\usepackage[unicode,bookmarks,colorlinks]{hyperref}
\hypersetup{
    linkcolor=brickred,
}

\definecolor{mahogany}{cmyk}{0, 0.77, 0.87, 0}
\definecolor{salmon}{cmyk}{0, 0.53, 0.38, 0}
\definecolor{melon}{cmyk}{0, 0.46, 0.50, 0}
\definecolor{yellowgreen}{cmyk}{0.44, 0, 0.74, 0}
\definecolor{brickred}{cmyk}{0, 0.89, 0.94, 0.28}
\definecolor{OliveGreen}{cmyk}{0.64, 0, 0.95, 0.40}
\definecolor{RawSienna}{cmyk}{0, 0.72, 1.0, 0.45}
\definecolor{ZurichRed}{rgb}{1, 0, 0} 

\usepackage[margin=1.6in]{geometry}

\numberwithin{equation}{section}

\newtheorem{thm}{Theorem}[section]        

\newtheorem{prop}{Proposition}[section]
\newtheorem{rmk}{Remark}[section]

\newcommand{\R}{\mathbb{R}}                  
\newcommand{\Rd}{\R^d}               
\newcommand{\ioRd}{\int_{\Rd}}    
\newcommand{\set}[1]{ \left\{#1\right\} }
\newcommand{\myprod}[3]{\prod\limits_{#1=#2}^{#3}}

\newcommand{\abs}[1]{\left|#1\right|}

\newcommand{\tgo}{t\rightarrow 0+}
\newcommand{\F} {(-\Delta)^{ \frac{\alpha}{2} } }
\newcommand{\Prob} {\mathbb{P} }
\newcommand{\palp}{p_{t}^{(\alpha)}}
 
\newcommand{\dom}{ \Omega}
\newcommand{\pthesis}[1]{\left(#1\right)}
\newcommand{\pint}[1]{\left[#1\right]}
\newcommand{\E}{\mathbb{E}}
\newcommand{\ind}{\mathbbm{1}_{\dom}}
\newcommand{\proc}{{\bf X}}
\newcommand{\Halp}{\mathbb{H}_{\Omega}^{(\alpha)}(t)}
\newcommand{\pf}{\Psi^{(\alpha)}_{\dom}(t)}
\newcommand{\lm}{\mathcal{L}(\Rd)}
\newcommand{\per}{Per(\dom)}
\newcommand{\ld}{\abs{\Omega}}
\newcommand{\mylim}[2]{\lim\limits_{#1}{#2}}

\newcommand{\myP}{\mathcal{P}}
\newcommand{\et}{\tau_{\dom}^{(\alpha)}}

\begin{document}

\title[]{Heat content estimates for the fractional Schr\"odinger operator $\F+\ind$.}
\author{Luis Acu\~na Valverde}
\address{Department of Mathematics, Universidad de Costa Rica, San Jos\'e, Costa Rica.}
\email{luisacunavalverde@ucr.ac.cr/ guillemp22@yahoo.com}

\maketitle


\begin{abstract} 
This paper establishes by employing  analytic and probabilistic techniques  estimates concerning the {\it heat content} for  the fractional Schr\"odinger operator  $\F+\ind$ with $0<\alpha\leq 2$ in $\Rd$, $d\geq 2$ and $\dom$ a Lebesgue  measure set satisfying some regularity conditions.
\end{abstract}

{\small {\bf Keywords:} isotropic stable processes, spectral heat content,  heat content, Schr\"odinger operator, Fractional Laplacian, set of  finite perimeter, $\alpha$--perimeter.} 

\section{introduction}

Let $0<\alpha\leq 2$ and consider $\proc=\set{X_t}_{t\geq0}$ a rotationally  invariant  $\alpha$-stable process in $\Rd$, $d\geq 2$, whose transition densities denoted by $p_t^{(\alpha)}(x,y)=\palp(x-y)$ are uniquely determined by their Fourier transform (characteristic function) and which are given by 
\begin{equation}\label{CharF.Proc}
e^{-t\abs{\xi}^{\alpha}}=\E_{0}[e^{-\dot{\iota} \xi\cdot X_t }]=\ioRd dy\,e^{-\dot{\iota}y\cdot\xi}\palp(y),
\end{equation}
for all $t>0$, $\xi\in\Rd$. Henceforth, $\mathbb{E}_x$ will represent the expectation of the process $\proc$ starting at $x$. One  fact about $\palp(x,y)$ to be used in our proofs and it can be deduced from the last identity is that for all $a>0$, $t>0$ and $z\in\Rd$, we have 
\begin{align}\label{basequ}
\int_{\Rd}dx\,p_{ta}^{(\alpha)}(x,z)=1.
\end{align}
It is  worth mentioning that the above heat kernels are {\it positive} and radial. Also, they have explicit expressions only for $\alpha=1$ and $\alpha=2$. In fact, for  $\alpha=2$, $\proc$ is a Brownian motion with transition density given by the Gaussian kernel. That is,
\begin{align*}
p_t^{(2)}(x,y)=(4\pi t)^{-d/2}e^{-\abs{x-y}/4t}.
\end{align*}
Regarding $\alpha=1$, $\proc$ corresponds to the Cauchy process with transition densities given by the Poisson kernel. Namely,
\begin{align*}
p^{(1)}_t(x,y)=\frac{k_d\,t}{\pthesis{t^2+\abs{x-y}^2}^{(d+1)/2}},
\end{align*}  
where
\begin{align}\label{kappadef}
k_d=\frac{\Gamma\pthesis{\frac{d+1}{2}}}{\pi^{\frac{d+1}{2}}}.
\end{align} 

In general and  for the purposes of this paper, we shall only require  the following two  facts about $\palp(x,y)$ for all $\alpha \in (0,2)$.  First, there exists $c_{\alpha,d}>0$ such that for all $x,y\in \R^d$ and $t>0$ (see \cite{Chen1} for more details), we have that
\begin{align}\label{tcomp}
c_{\alpha,d}^{-1}\min\set{t^{-d/\alpha},\frac{t}{\abs{x-y}^{d+\alpha}}}\leq
\palp(x-y)\leq c_{\alpha,d}\min\set{t^{-d/\alpha},\frac{t}{\abs{x-y}^{d+\alpha}}}.
\end{align}
  Secondly, according to \cite[Theorem 2.1]{Blum}, we have
\begin{align*}
\lim_{\tgo}\frac{\palp(x-y)}{t}=\frac{\beta_{\alpha,d}}{\abs{x-y}^{d+\alpha}},
\end{align*}
for all $x\neq y$, where 
\begin{equation}\label{Adef}
\beta_{\alpha,d}=\alpha \, 2^{\alpha-1}\, \pi^{-1-\frac{d}{2}}\,\sin\pthesis{\frac{\pi\alpha}{2}}\,\Gamma\pthesis{\frac{d+\alpha}{2}}\,\Gamma\pthesis{\frac{\alpha}{2}}.
\end{equation}

On the other hand,  it is well known that the linear operator $\F$ called the Fractional Laplacian corresponds to the infinitesimal generator of $\proc$, whereas the linear operator $\F +\ind$ is called  the  Fractional Schr\"{o}dinger operator with respect to the potential $\ind$. This last operator is self-adjoint and it can be defined  as the infinitesimal generator of the heat semigroup given by
$$\E_x[e^{-\int_{0}^{t}\,ds\ind(X_s)}f(X_t)],$$
for $f \in \mathcal{S}(\Rd)$(Schwartz class).  The  transition densities of the aforementioned  heat semigroup can be represented by means of the Feynman-Kac formula (see for example  \cite{AcuBa, Acu5, BanYil, Hir, Kal, Sim} for further details concerning infinitesimal generators and heat kernels representations) which is given by
\begin{equation}\label{Fey.Kac.f}
\palp(x,y)\mathbb{E}_{x,y}^t\left[e^{-\int_{0}^{t}ds\,
\ind(X_{s})}\right],
\end{equation} 
where  $\mathbb{E}_{x,y}^t$ stands for the expectation with respect to the stable bridge process  associated with $\proc$ starting at $x$ and conditioned to be at $y$ at time $t$. 

With the proper introduction of the heat kernels provided in the identities \eqref{CharF.Proc} and \eqref{Fey.Kac.f}, we proceed to present the object to be studied. In this paper, we shall investigate the  behavior of the following function
\begin{align}\label{Schdef}
\Psi^{(\alpha)}_{\dom}(t)=\int_{\Rd}dx\int_{\Rd}dy\,\palp(x,y)\pthesis{1-\E^{t}_{x,y}\pint{e^{-\int_{0}^{t}ds\,\mathbbm{1}_{\dom}(X_s)}}},
\end{align}
when $\tgo$ and where $\dom \subseteq \Rd$ stands for a Lebesgue measurable set  to satisfy additional conditions. Based on the definition \eqref{Schdef}, we observe that the main purpose  of $\pf$ is to describe how different are  the heat kernels $\palp(x,y)\E^{t}_{x,y}\pint{e^{-\int_{0}^{t}ds\,\mathbbm{1}_{\dom}(X_s)}}$ and $\palp(x,y)$   in average at time $t$. One of the main goals of this paper is to reflect how  geometric features about  $\dom$ intervene on the  small time asymptotic behavior of $\pf$. The function $\Psi^{(\alpha)}_{\dom}(t)$ is called {\it the heat content for the Schr\"odinger operator} $\F+\ind$, since  according to  \cite{AcuBa},  it resembles  {\it the heat
content} for $\dom$ with respect to the process $\proc$, 
which is defined by
\begin{align}\label{hcdef}
\Halp=\int_{\dom}dz_1\int_{\dom}dz_2\,\palp(z_1,z_2)=\int_{\dom}dz_1\, \Prob_{z_1}(X_t\in \dom).
\end{align}

In order to motivate our work, we will denote throughout the paper the set of all Lebesgue measurable sets of $\R^d$ by $\lm$ and for $\dom\in \lm$ bounded, we will use $\ld$ to represent its Lebesgue measure.

We now proceed to discuss about $\Halp$. We point out that $\Halp$ is a bounded function for any $\dom \in \lm$ of finite measure since by \eqref{basequ} we have
\begin{align}\label{boundhc}
0\leq \Halp\leq \int_{\dom}dz_1\int_{\Rd}dz_2 \,p_{t}^{(\alpha)}(z_1, z_2)=\int_{\dom}dz_1=\ld,
\end{align}
for all $t\geq 0$. Results concerning the small time asymptotic behavior of  $\Halp$ have been developed in the recent years (see for instance Theorem \ref{essencialtools} below  and \cite{AcuPoisson} for current developments on the heat content for the Cauchy process) and they shall be essential to obtain estimates for $\Psi_{\dom}^{(\alpha)}(t)$ so that we refer the interested reader to  \cite{Acu1, Acu, Acu4} regarding estimates for the heat content for stable processes and \cite{Cyg, Grz} for results concerning more general L\'evy processes.  These results are of great interest   because these asymptotic behaviors involve terms describing geometric features about the underlying set $\dom$ as volume, surface area, mean curvature, etc. Also, $\Halp$ has been used in \cite{Acu} to obtain estimates about the behavior of the second term as $\tgo$ for {\it the  spectral heat content} defined by
\begin{align}\label{defshc}
Q_{\dom}^{(\alpha)}(t)=\int_{\dom}dx\,\Prob_{x}\pthesis{X_s\in \dom, \forall s\leq t},
\end{align}
whose  explicit second order expansion   is still unknown   for $0<\alpha<2$ since for the Brownian motion  was proved  in  \cite{van2} for smooth bounded  domains $\dom\subseteq \R^d$, $d\geq 2$ that
\begin{equation*}\label{VandenberHeatContBM}
Q_{\dom}^{(2)}(t)=\ld-\frac{2}{\sqrt{\pi}}|\partial \Omega|t^{1/2}+\left(2^{-1}(d-1)\int_{\partial \Omega}M(s)ds\right) t +\mathcal{O}(t^{3/2}),  
\end{equation*}
as $t\downarrow 0$.  Here, $M(s)$ denotes the mean curvature at the point $s\in \partial \dom$.

 Before continuing, we shall make some remarks concerning {\it the heat content and the spectral heat content} to be needed later. We first notice that {\it the spectral heat content} defined in \eqref{defshc} can be rewritten in terms of the exit time of $\proc$ from $\dom$. That is, if we define the following random variable
 \begin{align}\label{et}
 \et=\inf\set{s\geq 0: X_s\in \dom^{c}},
 \end{align}
known as the first exit time from $\dom$, then
\begin{align}\label{et}
Q_{\dom}^{(\alpha)}(t)=\int_{\dom}dx\,\Prob_{x}\pthesis{t\leq \et},
 \end{align}
 since  the events $\set{t\leq \et}$ and $\set{X_s\in \dom, \forall s\leq t}$ are identical. Moreover, it follows   from  \eqref{et} that $Q_{\dom}^{(\alpha)}(t)$ is a non increasing function with $Q_{\dom}^{(\alpha)}(0)=\ld$ and 
 \begin{align}\label{ineqQH}
 Q_{\dom}^{(\alpha)}(t)\leq \Halp,
 \end{align}
 since $\set{X_s\in \dom, \forall s\leq t}\subseteq \set{X_t\in \dom}$.

It is noteworthy that the estimates obtained below for 
$\pf$ come basically from providing  an asymptotic expansion for  $\pf$, where the functions $\Halp$ and $Q_{\dom}^{(\alpha)}(t)$  are involved (see Theorem \ref{thmT} and Theorem \ref{lem2} below).  Hence, this last  fact highlights the  deep connection among  these functions.

It should be mentioned that another relevant reason of why to investigate  the function $\Psi_{\dom}^{(\alpha)}(t)$  is due to the fact that it is linked to a wide variety of areas in probability since the stochastic integral 
\begin{align*}
\int_{0}^{t}ds\ind(X_s),
\end{align*} 
which is just the total time spent in $\dom$ up to time
$t$ by the stable process $\proc$ appears  in  many  problems regarding   occupation  measures,  local times and conditional gauge theorems.

We now carry on with the presentation of our main result. To do so, we require the concept of finite perimeter for a Lebesgue measurable set which we proceed to define.  We say that a bounded set $\dom \in \lm$ has finite perimeter if 
\begin{align}\label{defper}
0\leq  \sup\set{\int_{\dom}dx\,{\bf div} \varphi(x): \varphi\in C^{1}_{c}(\Rd,\Rd), ||\varphi||_{\infty}\leq 1}<\infty
 \end{align}     
 and we denote the last quantity by $Per(\Omega)$. The quantity $Per(\Omega)$ is called the perimeter of $\dom$ and coincides with the surface area of $\dom$ provided that $\dom$ has smooth boundary. We refer the interested reader on specifics about the perimeter of a set to \cite{Evans, Galerne,  Mor}. 
 
 On the other hand, for $0<\alpha<1$ and  $\Omega\in \lm$ bounded set, the $\alpha$-perimeter is denoted by $\myP_{\alpha}(\dom)$ and defined by 
 \begin{align}\label{alphaper}
\myP_{\alpha}(\Omega)=\int_{\dom}dy\int_{\dom^c}\frac{dx}{|x-y|^{d+\alpha}}.
\end{align}
We remark  that  $\myP_{\alpha}(\Omega)<\infty$  provided that $\per <\infty$ and $0<\alpha<1$ according to Corollary 2.13 in \cite{Lau}. Moreover, $\myP_{\alpha}(\Omega)$  turns out to be linked with celebrated Hardy and isoperimetric inequalities.  We refer the  reader to the papers of Z. Q. Chen, R. Song \cite{Song} and R. L. Frank, R. Seiringer \cite{Frank} for  further results involving this quantity. In fact, it is shown in \cite{Frank} that there  exists $\lambda_{d,\alpha}>0$ such that
\begin{align*}
|\Omega|^{(d-\alpha)/d}\leq \lambda_{d,\alpha}\,\myP_{\alpha}\pthesis{\Omega},
\end{align*}
with equality if and only if $\Omega$ is a ball.
It is also proved in \cite{Lau} and \cite{FracP} that
\begin{align*}
\lim_{\alpha\downarrow 0}\alpha \myP_{\alpha}(\Omega)=d\,\,|B_1(0)|\,\,|\Omega|,\nonumber \\
\lim_{\alpha\uparrow 1}(1-\alpha)\myP_{\alpha}(\Omega)=K_{d}\,\,\per,
\end{align*}  
for some $K_d>0.$

With all the geometric objects properly introduced, we  continue with the statement of our main result, but before doing so, we point out that the main purpose of this result is to  tell us how fast $\Psi_{\dom}^{(\alpha)}(t)$ goes to zero in terms of an upper bound involving geometric properties of the set $\dom$.

\begin{thm}\label{mthm} Let $d\geq2$ be an integer. Consider $\dom \in \lm$ a bounded set with $\per<\infty$ and nonempty interior. 
\begin{enumerate}
\item[$(i)$] Consider  $\alpha\in (1,2]$ and set
$$c_{\alpha}^{*}=\frac{\alpha^2\,\Gamma\pthesis{1-\frac{1}{\alpha}}}{\pi(1+\alpha)(1+2\alpha)}.$$
Then, we have  for all $t>0$ that
\begin{align}\label{ineq4}
 \Psi_{\dom}^{(\alpha)}(t)+t\pthesis{\frac{t}{2}-1}\ld \leq \frac{\ld}{3!}t^{3}+c_{\alpha}^{*}\per t^{2+\frac{1}{\alpha}}
\end{align}
and
\begin{align}\label{lim4}
\lim\limits_{\tgo}\frac{\Psi_{\dom}^{(\alpha)}(t)+t\pthesis{\frac{t}{2}-1}\ld }{t^{2+\frac{1}{\alpha}}}=c_{\alpha}^{*}\per.
\end{align}
\item[$(ii)$] For $\alpha=1$, there exists a constant $\gamma_d(\dom)>0$ such that
\begin{align*}
\Psi_{\dom}^{(1)}(t)+t\pthesis{\frac{t}{2}-1}\ld  \leq  \gamma_d(\dom)t^{3}+\frac{\per}{3!\pi}t^{3}\ln\pthesis{\frac{1}{t}}
\end{align*}
provided that $0<t<\min\set{diam(\dom),e^{-1}}$ where $diam(\dom)=\sup\set{\abs{x-y}:x,y\in \dom}$.
As a result, 
\begin{align*}
\varlimsup\limits_{\tgo}\frac{\Psi_{\dom}^{(1)}(t)+t\pthesis{\frac{t}{2}-1}\ld }{t^{3}\ln\pthesis{\frac{1}{t}}}\leq \frac{\per}{3!\pi}.
\end{align*}

\item[$(iii)$] Consider $0<\alpha<1$. Then, for all $t>0$, we obtain that
\begin{align}\label{3rd} 
\Psi_{\dom}^{(\alpha)}(t)+t\pthesis{\frac{t}{2}-1}\ld \leq
\frac{t^3}{3!}\pthesis{\ld+c_{\alpha,d}\myP_{\alpha}(\Omega)}
\end{align}
with $\myP_{\alpha}(\Omega)$ and $c_{\alpha,d}$ as defined in \eqref{alphaper} and \eqref{tcomp}, respectively. Furthermore, 
\begin{align*}
\lim\limits_{\tgo}\frac{\Psi_{\dom}^{(\alpha)}(t)+t\pthesis{\frac{t}{2}-1}\ld}{t^{3}}=\frac{1}{3!}\pthesis{ \ld+\beta_{\alpha,d}\myP_{\alpha}(\Omega)}
\end{align*}
with $\beta_{\alpha,d}$  as defined in \eqref{Adef}.
\end{enumerate}
\end{thm}

The paper is organized as follows.  In \S \ref{sec:prelim}, we develop the main tools to be used in the proof of Theorem \ref{mthm}. Finally, in \S \ref{sec:proofthm}, we provide the proof of our main result by appealing to estimates about  $\Halp$  shown in \cite{Acu1}.

\section{preliminaries results}\label{sec:prelim}
In this section, we shall develop the main tools to prove Theorem \ref{mthm} by employing some of the techniques found in \cite{AcuBa} and \cite{BanYil}. 

We  begin  by  introducing  some notation to conveniently express our formulas below.  Consider $k\in \mathbb{N}$ and   $\dom\subseteq \R^d$. We set
\begin{align}\label{heavynotation}
\dom^{k}&=\set{(w_1,w_2,...,w_k): w_j\in \dom,\, j=1,...,k},\\ \nonumber
I_k &=\set{(\lambda_{1},\lambda_{2},...,\lambda_k)\in [0,1]^{k}: 0<\lambda_{1}<\lambda_{2}< ... < \lambda_{k}<1},
\end{align}
and  $dz^{(k)}=dz_kdz_{k-1}...dz_1$ for $z_j\in \Rd$, $j=1,...,k$. Also, we define $h_{\dom}^{(k)}:\R_{+}^{k}\rightarrow \R_+$
by
\begin{align}\label{hfunction}
h_{\dom}^{(k)}(t,a_1,...,a_{k-1})=\int_{\dom^k}dz^{(k)}\myprod{j}{1}{k-1}p_{ta_j}^{(\alpha)}(z_j,z_{j+1}).
\end{align}

\begin{rmk}
At this point, we note  that the key ingredient in our results is the knowledge about the behavior of $h_{\dom}^{(2)}(t,a)$ for $a>0$ fixed as $\tgo$ since this function is equal to the heat content $\mathbb{H}_{\Omega}^{(\alpha)}(at)$. However, in order to obtain better estimates for $\Psi_{\dom}^{(\alpha)}(t)$ is required to know 
how the functions $h_{\dom}^{(k)}(t)$ behave for small time provided that $k\geq 3$ and this is an open problem so far.
\end{rmk}
 We now recall that the finite dimensional distributions of the stable bridge  process  (see \cite{BanYil}, \cite{Ber} and references therein for details) are given by
\begin{align}\label{stablebridge}
&\Prob^t_{x,y}\left(X_{s_1} \in dz_1 , X_{s_2} \in dz_2 , . . . , X_{s_k} \in dz_k\right)\\\nonumber &=\frac{dz^{(k)}}{\palp(x, y)}p_{s_{1}}^{(\alpha)}(x,z_{1})\pthesis{\myprod{j}{1}{k-1}p_{s_{j+1}-s_j}^{(\alpha)}(z_j,z_{j+1})}p_{t-s_k}^{(\alpha)}(z_{k},y),
\end{align}
where  $0<s_1<...<s_k<t$ and $x,y\in \Rd$.

The following proposition shows that the function $\Psi_{\dom}^{(\alpha)}(t)$ is bounded provided that $\dom\subseteq \Rd$ is a Lebesgue measurable set of finite measure.

\begin{prop} \label{bl1} Let $\dom\in \lm$ have  finite measure. Then, for all $t\geq 0$, we have that
$$0\leq \Psi_{\dom}^{(\alpha)}(t)\leq t|\dom|.$$
\end{prop}
\begin{proof}
By appealing to the basic inequality $0\leq 1-e^{-w}\leq w$ for $w\geq 0$, we obtain that
\begin{align}\label{psi1}
0 \leq \Psi_{\dom}^{(\alpha)}(t)\leq \int_{\Rd}dx \int_{\Rd}dy\,\palp(x,y)\E^{t}_{x,y}\pint{\int_{0}^{t}ds\,\mathbbm{1}_{\dom}(X_s)}.
\end{align}

Now, by using Fubini's theorem and the finite distribution of the stable bridge provided in \eqref{stablebridge}, we arrive at
\begin{align*}
\E^{t}_{x,y}\pint{\int_{0}^{t}ds\,\mathbbm{1}_{\dom}(X_s)}&=
\int_{0}^{t}ds\,\E^{t}_{x,y}\pint{\mathbbm{1}_{\dom}(X_s)}=
\int_{0}^{t}ds\,\int_{\Rd}\,\Prob_{x,y}^{t}(X_s\in dz)\ind(z)
\\&=
\int_{0}^{t}ds\int_{\dom}dz\frac{p^{(\alpha)}_{s}(x,z)p^{(\alpha)}_{t-s}(z,y)}{\palp(x,y)}.
\end{align*}
Therefore, due to the last identity, the inequality \eqref{psi1} and Fubini's theorem, we conclude that
\begin{align*}
0 \leq \Psi_{\dom}^{(\alpha)}(t)\leq \int_{0}^{t}ds\int_{\dom}dz \int_{\Rd}dx\,p^{(\alpha)}_{s}(x,z)\int_{\Rd}dy\,p^{(\alpha)}_{t-s}(z,y)=t|\dom|,
\end{align*}
where we have used \eqref{basequ} to obtain the last identity.
\end{proof}

Our next result  consists in a generalization of a conclusion contained in  the proof of the last proposition. Additionally, it also shows us where {\it the spectral heat content } comes into play.
\begin{thm}\label{thmT}  Let $\dom\in \lm$ have  finite measure.  Consider $k\in \mathbbm{N}$ and $Q_{\dom}^{(\alpha)}(t)$ as defined in \eqref{defshc}. Define
\begin{align}\label{Tdef}
T_{\dom}^{(k)}(t)=\int_{\Rd}dx\int_{\Rd}dy \,\palp(x,y)\E^{t}_{x,y}\pint{\pthesis{\int_{0}^{t}ds \mathbbm{1}_{\dom}(X_s)}^{k}}.
\end{align}
Then, for all $t\geq 0$, we have that 
\begin{enumerate}
\item[$(a)$] 
$t^{k} Q_{\dom}^{(\alpha)}(t) \leq T_{\dom}^{(k)}(t)\leq t^{k}|\dom|$,
\\
\item[$(b)$] 

$$T_{\dom}^{(k)}(t) =k!\,t^k\int_{I_k}d\lambda^{(k)}h_{\dom}^{(k)}(t,\lambda_2-\lambda_1,...,\lambda_{k}-\lambda_{k-1})$$
with $I_k$ and $h_{\dom}^{(k)}$ as defined in \eqref{heavynotation}
and \eqref{hfunction}, respectively. 
\end{enumerate}
Furthermore,
\begin{align}\label{implimit}
\lim\limits_{\tgo}t^{-k}\,T^{(k)}_{\dom}(t)=\ld.
\end{align}
\end{thm}
\begin{proof} To begin with, we notice  that
$$\pthesis{\int_{0}^{t}ds \mathbbm{1}_{\dom}(X_s)}^{k}\leq t^{k-1}\int_{0}^{t}ds \mathbbm{1}_{\dom}(X_s)$$
which in turn implies by the definition of $T_{\dom}^{(k)}(t)$ in \eqref{Tdef} that
\begin{align}\label{inq2}
T_{\dom}^{(k)}(t)\leq t^{k-1}T_{\dom}^{(1)}(t).
\end{align}
Now, we remark that  the proof of Proposition \ref{bl1}  asserts that
$T_{\dom}^{(1)}(t)= t|\dom|$ so that together with \eqref{inq2}  yield the upper bound described in part $(a)$.

On the other hand, we recall the following identity established in \cite{Tib}, which says that
\begin{equation}\label{symmetric}
\left(\int_{0}^{1}du\,\tilde{V}(u)\right)^k= k!\int_{I_k}d\lambda^{(k)}\,\myprod{j}{1}{k}\tilde{V}(\lambda_j),
\end{equation}
for any $\tilde{V}:[0,1]\rightarrow \R_+$ Lebesgue measurable positive function. Next, by  observing that $$\int_{0}^{t}ds \mathbbm{1}_{\dom}(X_s)=t\int_{0}^{1}du\mathbbm{1}_{\dom}(X_{tu}),$$  we arrive due to an application of  $\eqref{symmetric}$ with $\tilde{V}(u)=\mathbbm{1}_{\dom}(X_{tu})$  at
\begin{equation}
\left(\int_{0}^t ds\,\mathbbm{1}_{\dom}(X_{s})\right)^k= k!\,t^k\,\int_{I_k}d\lambda^{(k)}\myprod{j}{1}{k}\mathbbm{1}_{\dom}(X_{t\lambda_j}).
\end{equation}
Thus, because of the formula \eqref{stablebridge}, we obtain that $$\E^{t}_{x,y}\pint{\myprod{j}{1}{k}\mathbbm{1}_{\dom}(X_{t\lambda_j})}=\int_{\R^{kd}}\Prob^{t}_{x,y}\pthesis{X_{t\lambda_1}\in dz_1,...,X_{t\lambda_k}\in dz_k}\myprod{j}{1}{k}\mathbbm{1}_{\dom}(z_j)$$ is also equal to

\begin{align}\label{ineq2}
\int_{\dom^{k}}\frac{dz^{(k)}}{\palp(x,y)}p_{t\lambda_{1}}^{(\alpha)}(x,z_{1})\pthesis{\myprod{j}{1}{k-1}p_{t(\lambda_{j+1}-\lambda_j)}^{(\alpha)}(z_{j+1},z_j)}p_{t(1-\lambda_k)}^{(\alpha)}(z_{k},y).
\end{align}

Next, it  follows from  Fubini's theorem and \eqref{Tdef} that $T_{\dom}^{(k)}(t)$ can be rewritten as
$$k!\,t^k\int_{I_k}d\lambda^{(k)}\int_{\Rd}dx\int_{\Rd}dy\,\palp(x,y)\E^{t}_{x,y}\pint{\myprod{j}{1}{k}\mathbbm{1}_{\dom}(X_{t\lambda_j})}.$$
Therefore, by appealing to \eqref{ineq2}, \eqref{basequ} and Fubini's theorem again, we obtain that $T_{\dom}^{(k)}(t)$ is equal to 
\begin{align*}
&k!\,t^k\int_{I_k}d\lambda^{(k)}\int_{\R^{d}}dx\,p^{(\alpha)}_{t\lambda_1}(x,z_1)\int_{\dom^{k}}dz^{(k)}\,\myprod{j}{1}{k-1}p_{t(\lambda_{j+1}-\lambda_j)}^{(\alpha)}(z_{j+1},z_j)\int_{\Rd}dy\,p_{t(1-\lambda_{k})}^{(\alpha)}(z_k,y)\\ \nonumber
&=k!t^k \int_{I_k}d\lambda^{(k)}h_{\dom}^{(k)}(t,\lambda_2-\lambda_1,...,\lambda_k-\lambda_{k-1})
\end{align*}
with  $I_k$ and $h_{\dom}^{(k)}$ as defined in \eqref{heavynotation}
and \eqref{hfunction}, respectively. Thus, we conclude part $(b)$.

 In order to  obtain the lower bound in part $(a)$, we consider the exit time $\et$ defined in \eqref{et} and notice that the following expectation $\E^{t}_{x,y}\pint{\pthesis{\int_{0}^{t}ds\,\ind(X_s)}^{k}}$ which can be regarded as
\begin{align*}
\E^{t}_{x,y}\pint{\pthesis{\int_{0}^{t}ds\,\ind(X_s)}^{k}; t\leq \et}+\E^{t}_{x,y}\pint{\pthesis{\int_{0}^{t}ds\,\ind(X_s)}^{k};t> \et}
\end{align*}
is bounded below by 
\begin{align*} 
\E^{t}_{x,y}\pint{\pthesis{\int_{0}^{t}ds\,\ind(X_s)}^{k}; t\leq \et}= t^{k}\,\Prob^{t}_{x,y}\pthesis{t\leq \et}
\end{align*}
where the last identity is a consequence of the fact that $\int_{0}^{t}ds\,\ind(X_s)=t$ under the event $\set{t\leq \et}=\set{X_s\in \dom, \forall s\leq t}$. Therefore, based on our previous estimates and definitions, we arrive at
\begin{align}\label{ineimp2}
t^{-k}T^{(k)}_{\dom}(t)&\geq \, \int_{\Rd}dx\int_{\Rd}dy\, \palp(x,y)\,\Prob^{t}_{x,y}\pthesis{t\leq \et}\\ \nonumber
&\geq \int_{\dom}dx\int_{\dom}dy\, \palp(x,y)\,\Prob^{t}_{x,y}\pthesis{t\leq \et}.
\end{align}
 Now, since $X_t$  has a density which is  absolutely continuous with respect to the Lebesgue measure, it follows by appealing to conditional probabilities that
 \begin{align*}
 \Prob_x\pthesis{t\leq \et}&=\int_{\Rd}dy\, \Prob_x\pthesis{t\leq \et |X_t=y}\palp(x,y)\\ \nonumber
 &=\int_{\dom}dy \,\Prob_{x,y}^{t}\pthesis{t\leq \et }\palp(x,y).
 \end{align*}
Hence, by combining \eqref{ineimp2} and \eqref{defshc} together with the last identity, we have shown that $t^{-k}T^{(k)}_{\dom}(t)\geq Q_{\dom}^{(\alpha)}(t)$. This finishes the proof of part $(a)$.

To conclude \eqref{implimit} and the proof of the theorem, it suffices to show due to part $(a)$ that $\varliminf\limits_{\tgo}Q_{\dom}^{(\alpha)}(t)\geq \ld$. To prove this, we turn to Fatou's Lemma and \eqref{defshc} to obtain that
$$\varliminf\limits_{\tgo}Q_{\dom}^{(\alpha)}(t)\geq \int_{\dom}dx \,\varliminf\limits_{\tgo}\Prob_x(t\leq \et)=\int_{\dom}dx\, \Prob_x(0\leq \et) = \int_{\dom}dx =\ld.$$

\end{proof}

The following result shows us how the function $\Psi_{\dom}^{(\alpha)}(t)$ is related to the heat content of $\dom$ 
with respect to the stable process $\proc$.

\begin{thm}\label{lem2} Let $\dom\in \lm$ have finite measure and consider the heat content $\Halp$ defined in \eqref{hcdef}. Then, there exists a nonnegative function $R(t)$ such that for all $t\geq 0$, we have that
\begin{align}\label{identity1}
\Psi_{\dom}^{(\alpha)}(t)+t\pthesis{\frac{t}{2}-1}\ld=t^2\int_{0}^{1}d\lambda_1\int_{\lambda_1}^{1}d\lambda_2\,\pthesis{\ld-\mathbb{H}_{\Omega}^{(\alpha)}(t(\lambda_2-\lambda_1))}+R(t).
\end{align}
Moreover,
\begin{enumerate}
\item[$(i)$]for all $t\geq 0$, we have 
\begin{align*}
\frac{e^{-t}}{3!}T^{(3)}_{\dom}(t)\leq  R(t)\leq \frac{\abs{\dom}}{3!}t^3,
\end{align*}
with $T^{(3)}_{\dom}(t)$ as defined in Theorem \ref{thmT} and

\item[$(ii)$] 
\begin{align*}
\lim\limits_{\tgo}\frac{R(t)}{t^3}=\frac{\ld}{3!}.
\end{align*}

\end{enumerate}

\end{thm}
\begin{proof}
By using the well known Taylor expansion of the exponential function of order 2, we have that $1-e^{-x}=x-\frac{1}{2}x^2 + e^{-\theta_x}\frac{x^3}{3!}$ for some $0<\theta_x<x$. Hence, this expansion combined with \eqref{Schdef} and \eqref{Tdef} yield
\begin{align}\label{p1}
\Psi_{\dom}^{(\alpha)}(t)=T_{\dom}^{(1)}(t)-\frac{1}{2}T_{\dom}^{(2)}(t)+R(t),
\end{align}
where
\begin{align*}
R(t)=\frac{1}{3!}\int_{\Rd}dx\int_{\Rd}dy \,\palp(x,y)\E^{t}_{x,y}\pint{e^{-\Theta_t}\pthesis{\int_{0}^{t}ds \mathbbm{1}_{\dom}(X_s)}^{3}}
\end{align*}
with $\Theta_t$ a random variable satisfying $$0<\Theta_t<\int_{0}^{t}ds\ind(X_s)\leq t.$$ Thus, part $(i)$ follows from the fact that $e^{-t}\leq e^{-\Theta_t}\leq 1$, the definition of $R(t)$ and Theorem \ref{thmT}.

On the other hand,  $(ii)$ is easily obtained by combining part $(i)$ together with the results established  in Theorem \ref{thmT}.

Finally, we observe that $T_{\dom}^{(1)}(t)=t|\dom|$ and
\begin{align*}
T_{\dom}^{(2)}(t)=2!t^2 \int_{0}^{1}d\lambda_1\int_{\lambda_1}^{1}d\lambda_2\,h_{\dom}^{(2)}(t,\lambda_2-\lambda_1).
\end{align*}
Now,  notice that \eqref{hfunction}  and \eqref{hcdef} imply 
$h_{\dom}^{(2)}(t,\lambda_2-\lambda_1)=\mathbb{H}_{\Omega}^{(\alpha)}(t(\lambda_2-\lambda_1))$
so that \eqref{p1} and our previous estimates allow us to conclude that
$$\Psi_{\dom}^{(\alpha)}(t)=t\ld-t^2\int_{0}^{1}d\lambda_1\int_{\lambda_1}^{1}d\lambda_2\,\mathbb{H}_{\Omega}^{(\alpha)}(t(\lambda_2-\lambda_1))+R(t).$$
Therefore, the desired identity \eqref{identity1} is obtained by adding at both sides  of the last expression the term
$$\frac{\ld}{2}t^2=t^2\int_{0}^{1}d\lambda_1\int_{\lambda_1}^{1}d\lambda_2\,\ld.$$
\end{proof}

\section{proof of theorem \ref{mthm}}\label{sec:proofthm}

The proof of our main result relies on the following theorem regarding estimates for the heat content $\Halp$ whose proofs can be found in  \cite{Acu1}. 

\begin{thm}\label{essencialtools} 
Let $d\geq 2$ be an integer and consider $\dom \in \lm$ a bounded set with nonempty interior satisfying that $\per<\infty$.
\begin{enumerate}
\item[$ (a)$]Let $\alpha\in (1,2]$. Then, the following inequality holds for all $t>0$.
\begin{align}\label{ineq1,2}
\ld- \Halp\leq\frac{ t^{\frac{1}{\alpha}}}{\pi}\,\Gamma\pthesis{1-\frac{1}{\alpha}} \per.
\end{align}
Furthermore,
\begin{align}\label{lim1,2}
\mylim{\tgo}{\frac{\ld- \Halp}{t^{\frac{1}{\alpha}}}}=\frac{1}{\pi}\Gamma\pthesis{1-\frac{1}{\alpha}}\per.
\end{align}

\item[$(b)$] For $\alpha=1$, we have for all  $0<t<\min\set{diam(\dom),e^{-1}}$ that
\begin{align}\label{c1}
\ld- \mathbb{H}_{\Omega}^{(1)}(t)\leq \lambda(\dom)t+\frac{1}{\pi}\per
t\ln\pthesis{\frac{1}{t}}.
\end{align}
Here,
$$\lambda(\dom)=\frac{\ld\,A_d\,\kappa_{d}}{diam(\dom)}+ \frac{\per}{\pi}\pthesis{\ln(diam(\dom))+\int_{0}^{1}\frac{dr\,r^d}{(1+r^2)^\frac{d+1}{2}}},$$
 $\kappa_d$ as given in \eqref{kappadef} and $A_d$ is the surface area of the unit ball in $\Rd$.
In particular, we arrive at
\begin{align}\label{limsup}
\varlimsup_{\tgo}\frac{\ld- \mathbb{H}^{(1)}_{\Omega}(t)}{t\ln\pthesis{\frac{1}{t}}}\leq \frac{1}{\pi}\per.
\end{align}

\item[$(c)$] For $0<\alpha<1$, 
\begin{align}\label{oto}
\mylim{\tgo}{\frac{\ld- \Halp}{t}}=
\beta_{\alpha,d}\,\,\myP_{\alpha}(\Omega),
\end{align}
\end{enumerate}
with $\myP_{\alpha}(\Omega)$ and $\beta_{\alpha,d}$ as provided  in \eqref{alphaper} and \eqref{Adef}, respectively.
\end{thm}

We now proceed with the proof of Theorem \ref{mthm}. 
\\

{\bf Proof of part $(i)$:}
By Theorem \ref{lem2} and \eqref{ineq1,2}, we deduce for all $t>0$  that
\begin{align*}
\Psi_{\dom}^{(\alpha)}(t)+t\pthesis{\frac{t}{2}-1}\ld\leq 
\frac{1}{\pi}\Gamma\pthesis{1-\frac{1}{\alpha}}t^{2+\frac{1}{\alpha}}\per\int_{0}^{1}d\lambda_1\int_{\lambda_1}^{1}d\lambda_2\,\pthesis{\lambda_2-\lambda_1}^{\frac{1}{\alpha}}+\frac{t^3}{3!}\ld.
\end{align*}
Thus,  \eqref{ineq4} follows from above inequality and the fact that
$$\int_{0}^{1}d\lambda_1\int_{\lambda_1}^{1}d\lambda_2\,\pthesis{\lambda_2-\lambda_1}^{\frac{1}{\alpha}}=\frac{\alpha^{2}}{(1+\alpha)(1+2\alpha)}.$$

Now,  it is easy to see that \eqref{ineq4} and the fact that $1<\alpha\leq 2$ imply 
$$\varlimsup\limits_{\tgo}\frac{\Psi_{\dom}^{(\alpha)}(t)+t\pthesis{\frac{t}{2}-1}\ld}{t^{2+\frac{1}{\alpha}}}\leq \varlimsup\limits_{\tgo}\pthesis{\frac{\ld}{3!}t^{1-\frac{1}{\alpha}} +c_{\alpha}^{*}\per}=c_{\alpha}^{*}\per.$$
Next, we notice that the right hand side of \eqref{identity1} is a sum of two nonnegative terms because of \eqref{boundhc} and Theorem \ref{lem2} so that  an application of Fatou's Theorem to \eqref{identity1} yields that
\begin{align*}
\varliminf\limits_{\tgo}\frac{\Psi_{\dom}^{(\alpha)}(t)+t\pthesis{\frac{t}{2}-1}\ld}{t^{2+\frac{1}{\alpha}}}&\geq \int_{0}^{1}d\lambda_1\int_{\lambda_1}^{1}d\lambda_2\,(\lambda_2-\lambda_1)^{\frac{1}{\alpha}}\varliminf\limits_{\tgo}\frac{\ld-\mathbb{H}_{\Omega}^{(\alpha)}(t(\lambda_2-\lambda_1))}{\pthesis{t(\lambda_2-\lambda_1)}^{\frac{1}{\alpha}}}\\ \nonumber
&=c_{\alpha}^{*}\per,
\end{align*}
where the last equality comes from the limit given in \eqref{lim1,2} and our previous computations. Therefore, the limit \eqref{lim4} holds true.
\\

{\bf Proof of part $(ii)$:} we start by providing the values of the following integrals to be needed below.
\begin{align}\label{basicint}
\int_{I_2}d\lambda^{(2)}(\lambda_2-\lambda_1)&=\frac{1}{3!},\\ \nonumber
\int_{I_2}d\lambda^{(2)}(\lambda_2-\lambda_1)\ln(\lambda_2-\lambda_1)&=-\frac{5}{36}.
\end{align}

Now, by \eqref{c1} we have for  $\lambda_2-\lambda_1>0$ that $\ld- \mathbb{H}_{\Omega}^{(1)}(t(\lambda_2-\lambda_1))$ is bounded above by
\begin{align*} 
\lambda(\dom)(\lambda_2-\lambda_1)t+\frac{1}{\pi}\per
t(\lambda_2-\lambda_1)\ln\pthesis{\frac{1}{t(\lambda_2-\lambda_1)}}.
\end{align*}
Next, due to the basic properties of the logarithmic function, we observe that the last expression can be written as
\begin{align*}
\pint{\lambda(\dom)(\lambda_2-\lambda_1)-\frac{\per}{\pi}(\lambda_2-\lambda_1)\ln\pthesis{\lambda_2-\lambda_1}}t+\frac{\per}{\pi}(\lambda_2-\lambda_1)\,t\ln\pthesis{\frac{1}{t}}.
\end{align*}
Consequently, it follows from Theorem \ref{lem2} and \eqref{basicint} that $\Psi_{\dom}^{(1)}(t)+t\pthesis{\frac{t}{2}-1}\ld$ is bounded above by
\begin{align*}
\pthesis{\frac{\ld+\lambda(\dom)}{3!}+\frac{5}{36\pi}\per}t^3+
\frac{\per}{3!\pi}t^{3}\ln\pthesis{\frac{1}{t}}.
\end{align*}
Therefore, the desired conclusion easily follows from our previous  estimates by taking
$$\gamma_{d}(\dom)=\frac{\ld+\lambda(\dom)}{3!}+\frac{5}{36\pi}\per.$$

{\bf Proof of part $(iii)$:} we  observe that \eqref{basequ} implies for all $z_1\in \Rd$
\begin{align*}
1=\int_{\Rd}dz_2\,\palp(z_1,z_2)=\int_{\dom}dz_2\,\palp(z_1,z_2)+\int_{\dom^c}dz_2\,\palp(z_1,z_2).
\end{align*}
Thus, because of the definition of the heat content given in \eqref{hcdef}, we arrive at
\begin{align*}
\ld-\Halp=\int_{\dom}dz_1\int_{\dom^{c}}dz_2\,\palp(z_1,z_2).
\end{align*}
Next, by appealing to inequality \eqref{tcomp}, the last identity and \eqref{alphaper}, we obtain that $\ld-\mathbb{H}_{\Omega}^{(\alpha)}(t(\lambda_2-\lambda_1))$ is bounded above by 
\begin{align}\label{ineq10}
t\,c_{\alpha,d}(\lambda_2-\lambda_1)\,\int_{\dom}dz_1\int_{\dom^{c}}\frac{dz_2}{\abs{z_1-z_2}^{d+\alpha}}=t\,c_{\alpha,d}(\lambda_2-\lambda_1)\,\myP_{\alpha}(\Omega)
\end{align}
for all $t>0$ and $\lambda_2-\lambda_1>0$.  Notice that the function $\lambda_2-\lambda_1\in L^1(I_2)$ by \eqref{basicint} and
$$\lim\limits_{\tgo}\frac{\ld-\mathbb{H}_{\Omega}^{(\alpha)}(t(\lambda_2-\lambda_1))}{t}=\beta_{\alpha,d}(\lambda_2-\lambda_1)\,\myP_{\alpha}(\Omega)$$
so that an application of Lebesgue dominated convergence Theorem shows that
\begin{align*}
\lim\limits_{\tgo}\int_{0}^{1}d\lambda_1\int_{\lambda_1}^{1}d\lambda_2\,\pthesis{\frac{\ld-\mathbb{H}_{\Omega}^{(\alpha)}(t(\lambda_2-\lambda_1))}{t}}= \frac{\beta_{\alpha,d}}{3!}\,\myP_{\alpha}(\Omega).
\end{align*}
Thus, by Theorem \ref{lem2}, we deduce that
\begin{align*}
\lim\limits_{\tgo}\frac{\Psi_{\dom}^{(\alpha)}(t)+t\pthesis{\frac{t}{2}-1}\ld}{t^3}&=\lim\limits_{\tgo}\int_{0}^{1}d\lambda_1\int_{\lambda_1}^{1}d\lambda_2\,\pthesis{\frac{\ld-\mathbb{H}_{\Omega}^{(\alpha)}(t(\lambda_2-\lambda_1))}{t}}+\frac{R(t)}{3!}\\ 
&= \frac{1}{3!}\pthesis{\beta_{\alpha,d}\,\myP_{\alpha}(\Omega)+\ld}.
\end{align*}

Finally, Theorem \ref{lem2} and \eqref{ineq10} show that
\begin{align*}
\Psi_{\dom}^{(\alpha)}(t)+t\pthesis{\frac{t}{2}-1}\ld \leq \frac{1}{3!}\pthesis{c_{\alpha,d}\,\myP_{\alpha}(\Omega)+\ld}
\end{align*}
and this completes the proof of our main result.
\\
\\
{\bf Acknowledgement}: This project has been supported by Universidad de Costa Rica, project 1976.

\end{document}